\documentclass[reqno]{amsart}

\usepackage{hyperref}

\numberwithin{equation}{section}
\newtheorem{theorem}{Theorem}
\newtheorem{lemma}[theorem]{Lemma}
\newtheorem{definition}[theorem]{Definition}
\newtheorem{remark}[theorem]{Remark}

\usepackage{color}
\usepackage{zref-totpages}

\allowdisplaybreaks

\begin{document}

\title[Existence result of the global attractor]{Existence result of the global attractor\\ 
for a triply nonlinear thermistor problem}

\thanks{\bf This is a \ztotpages\ pages preprint of a paper whose final and definite form 
is published in 'Moroccan J. of Pure and Appl. Anal. (MJPAA)', 
ISSN: Online 2351-8227 -- Print 2605-6364.}

\author[M. R. Sidi Ammi, I. Dahi, A. El Hachimi, D. F. M. Torres]{%
Moulay Rchid Sidi Ammi 
\and Ibrahim Dahi \and \\
Abderrahmane El Hachimi
\and Delfim F. M. Torres}

\address{Moulay Rchid Sidi Ammi (corresponding author) \newline
Department of Mathematics, AMNEA Group, MAIS Laboratory, 
Faculty of Sciences and Technics, 
Moulay Ismail B. P. 509, Errachidia, Morocco.}
\email{rachidsidiammi@yahoo.fr}

 \address{Ibrahim DAHI \newline
Department of Mathematics, AMNEA Group, MAIS Laboratory, 
Faculty of Sciences and Technics, 
Moulay Ismail B. P. 509, Errachidia, Morocco.}
\email{i.dahi@edu.umi.ac.ma}

\address{Abderrahmane El Hachimi \newline
Department of Mathematics, Faculty of Sciences,
Mohammed V University of Rabat, Morocco.}
\email{aelahacimi@yahoo.fr}

\address{Delfim F. M. Torres \newline
R\&D Unit CIDMA, Department of Mathematics,
University of Aveiro, 3810-193 Aveiro, Portugal.}
\email{delfim@ua.pt}

\subjclass[2010]{35A01, 35A02, 46E35}

\keywords{Existence; uniqueness; thermistor problem; Sobolev spaces; 
global attractor; $\omega-$limit; invariant set; absobsing set; semi-group.}

\begin{abstract}
We study the existence and uniqueness of a bounded weak solution 
for a triply  nonlinear thermistor problem in Sobolev spaces. 
Furthermore, we prove the existence of an absorbing set and, 
consequently, the universal attractor.
\end{abstract}

\maketitle

% -----------------------------------------

\section{Introduction}
\label{section1}

The thermistor was discovered by Michael Faraday in 1833, 
who noticed that the temperature increases when the silver sulfides 
resistance decreases. A lot of studies of the thermistor problem can be found in 
\cite{agarwal2021existence,el2004thermistor,hachimi2005existence,hrynkiv2020optimal,lacey1995thermal}.   

A thermistor is a circuit component that may be used as a current limiter 
or as a temperature sensor. It is, typically, a tiny cylinder, constructed 
of a ceramic substance whose electrical conductivity is highly dependent 
on temperature. The thermistor regulates the heat created by an electrical 
current traveling through a conductor device. Thermistor problems have received 
a lot of attention. We refer the reader to 
\cite{antontsev1994thermistor,cimatti1992existence,hachimi2005existence,g,lacey1995thermal,nanwate2022well}
and references therein. 

Thermistors are commonly used as temperature control devices 
in a wide variety of industrial equipment, ranging from space 
vehicles to air conditioning controllers. They are also often used in the
medical field, for localized and general body temperature measurement, 
in meteorology, for weather forecasting, and in chemical industries 
as process temperature sensors. A detailed description of thermistors 
and their applications in electronics and other industries can be found 
in \cite{ammi2017galerkin}.

There are two types of thermistors: NTC and PTC, which have a positive 
and negative temperature coefficient, respectively. An NTC thermistor 
is a temperature sensor that measures temperature using the resistance 
qualities of ceramic and metal composites. NTC sensors provide a number 
of benefits in terms of temperature sensing, including small size, 
great long-term stability, and high accuracy and precision.
The operation of a PTC electric surge device is as follows: when the circuit's 
current is suddenly increased, the device heats up, causing a dramatic decline 
in its electrical conductivity, effectively shutting off the circuit.
In this paper, we consider the following general nonlocal thermistor problem:
\begin{equation}
\label{(P)}
\quad \left\lbrace 
\begin{array}{lccccc}
\displaystyle \frac{\partial \alpha (v)}{\partial s}-\Delta_m v
=\kappa \frac{f(v)}{(\int_{\Omega}f(v)dx)^{2}},~&\mbox{in}&\quad Q,\\
\displaystyle \alpha(v(x,0))=\alpha (v_0),
&\mbox{ in }&\quad \Omega,\\
\displaystyle v= 0, &\mbox{ on }& \quad \Gamma\times]0,M[.
\end{array}
\right.  
\end{equation}
Problem $\eqref{(P)}$ models the diffusion of the temperature produced  
when an electric current flows crossing a material, where $f(v)$ is the 
electrical resistance of the conductor and 
$\displaystyle{\frac{f(v)}{(\int_{\Omega}f(v)dx)^{2}}}$ represents 
the non-local term of ${(\ref{(P)})}$. Here, $Q=\Omega\times [0,M]$,  
where $\Omega$ is an open bounded subset of $\mathbb{R}^N,~N\ge 1$, 
and  $M$ is a positive constant.

Problem \eqref{(P)} is a generalization of the problem  appearing in the work 
of Kavallaris and Nadzieja \cite{kavallarisblow}. For $\alpha(v)=v$ and $m=2$, 
one gets the classical model of the thermistor problem appearing in the work 
of Lacey \cite{lacey1995thermal}, which is a transformation of the following problem: 
\begin{equation}
\label{191916}
\begin{gathered}\displaystyle
\frac{\partial v}{\partial s}=\nabla \cdot(\kappa(v) \nabla v)+\rho(v)|\nabla \psi|^{2},\\
\displaystyle \quad\quad\quad\nabla \cdot(\rho(v) \nabla \psi)=0, 
\end{gathered}
\end{equation}
where $\kappa$ is the thermal conductivity, $\psi$ is the electrical potential, 
and $\rho(v)$ represents the electrical conductivity, which is normally 
a positive function supposed to drop sharply by several orders of magnitude 
at some critical temperature, and remains essentially zero for larger temperatures. 
This feature is essential for the intended functioning of thermistors 
as thermoelectric switches.

In the case $\alpha (v)=v$ and $m=2$, existence  and uniqueness results
of bounded weak solutions to problem  ${(\ref{(P)})}$ were  established 
in \cite{hachimi2005existence}.  Existence of an optimal control has 
been obtained by many authors with different assumptions on $f$ and $m$. 
We refer, for instance, to \cite{homberg2010optimal}. On the other hand, 
numerical computations of (\ref{(P)}) and (\ref{191916}) have been carried 
out by other authors, see for example 
\cite{ccatal2004numerical,ammi2008numerical,ammi2012optimal,zhou1997numerical}, 
in which the chosen parameters correspond to actual devices. Moreover,  
a study of \eqref{191916} in the case $N = 1$ can be found in \cite{montesinos2002evolution}. 
Here, we extend the existing literature of the nonlocal thermistor problem 
to a triply nonlinear case.

Let $B$ be the area of $\Omega$, $I$ the current such that 
$\displaystyle{\kappa=I^2/B^2}$, and $\Delta_m$ be defined by 	 
$$
\displaystyle{\Delta_m v=\operatorname{div}(\mid \nabla v\mid^{m-2} \nabla v)} 
\;\; \forall m\geq 2.
$$
We further specify the terms in $\eqref{(P)}$. We assume:
\begin{itemize}
\item[(H1)] $\displaystyle{v_0\in L^\infty(\Omega)}$;
	
\item[(H2)] $\displaystyle{\alpha:\mathbb{R}\longrightarrow \mathbb{R}}$
is a Lipschitz continuous increasing function such that 
$\displaystyle{\alpha (0)=0}$ and $\displaystyle{\alpha'(s)\geq\lambda>0}$ 
for all $\displaystyle{s\in \mathbb{R}}$;
	
\item[(H3)] $f$ is a Lipshitz continuous function, 
with compact support, verifying
$$
\displaystyle{\label{A1} \sigma\leq f(s), 
\mbox{ for all }s\in\mathbb{R}, \mbox{ for a positive constant  }\sigma}.
$$
\end{itemize}

The rest of the paper is organized as follows. 
In Section~\ref{section2}, we collect some basic concepts 
and a few known results that are useful to our development. 
Section~\ref{section3} is devoted to the existence of a classical 
solution to the regularized problem of ${(\ref{(P)})}$.
In Section~\ref{section4}, existence of a bounded weak solution 
to the regularized problem is proved. Then,
in Section~\ref{section55}, we provide sufficient conditions 
under which the solution is unique. Existence of an absorbing set, 
as well as the global attractor, are proved in Section~\ref{section6}.
Finally, we present some concluding remarks in Section~\ref{section7}.

% -----------------------------------------
  
\section{Preliminaries}
\label{section2}

In this section we collect a few known results that are useful to us.

\begin{definition}[See \cite{blanchard1988study}]
Let $\displaystyle{\alpha}$ be a continuous increasing function 
with $\displaystyle{\alpha(0)=0}$. For $\displaystyle{s\in\mathbb{R}}$ 
we define
$$
\displaystyle{\varPsi\left( s \right)=\int_{0}^{s} \alpha(t)dt}.
$$
The Legendre transform $\displaystyle{\varPsi^*}$ 
of $\displaystyle{\varPsi}$ is defined by
\begin{equation}
\displaystyle{\varPsi^*\left(t\right)
=\sup\limits_{r\in\mathbb{R}}\lbrace rt-\varPsi(t)\rbrace.}
\end{equation}
In particular, we get
\begin{equation}
\label{B2}
\displaystyle{
\varPsi^*\left(\alpha (t)\right)
=t\alpha (t)-\varPsi\left( t \right).}
\end{equation}
\end{definition}

\begin{remark}
\label{yy}
If $\displaystyle{v\in L^\infty(Q) }$, then 
$\displaystyle{\alpha(v)\in L^\infty(Q)}$.	
It turns out, from equality (\ref{B2}),   
that $\displaystyle{\varPsi^*\left(\alpha (v)\right)}$ is also bounded.
\end{remark}

\begin{lemma}[See \cite{Temam}]
\label{Lk}
Assume that $z$ is a non-negative, absolutely continuous function, 
satisfying the following inequality: 
$$
\displaystyle{z'(s)\leq hz(s)+g(s),~~~~\text{for}~~s\geq s_0,}
$$
where $h$ and $g$ are two non-negative integrable functions on 
$\displaystyle{\left[0,M\right]}$. Then, for each $s\in[0,M]$,
$$
\displaystyle{z(s)\leq\exp\left( \int_{0}^{s} h(\tau)d\tau \right)
\cdot\left[z(0)+ \int_{0}^{s} g(\tau)d\tau\right]}.
$$
\end{lemma}

\begin{lemma}[Ghidaglia lemma \cite{Temam}] 
\label{li} 
Let $z$ be a positive and absolutely continuous function 
on $]0,\infty[$ such that the inequality 
$$
z^{\prime}+\delta z^{q} \leq \eta
$$  
holds, where $q>1$, $\delta>0$, $\eta \geq 0$. Then,
$$
z(s) \leq\left(\frac{\eta}{\delta}\right)^{1 / q}+(\delta(q-1) s)^{-1 /(q-1)}
$$
for all $s \geq 0$.
\end{lemma}

\begin{lemma}[See \cite{alt1983quasilinear}] 
If $\displaystyle{v\in L^{m}\left(0,M;W^{1,m}(\Omega)\right)}$ 
with 
$$
\displaystyle{\dfrac{\partial\alpha (v)}{\partial s}
\in L^{m'}\left(0,M;W^{-1,m'}(\Omega))\right)},
$$ 
then
$$
\displaystyle{
\left\langle\dfrac{\partial\alpha (v)}{\partial s}, 
v\right\rangle_{W^{-1,m'}(\Omega),W^{1,m}(\Omega)}
=\frac{d}{d s} \int_{\Omega} \varPsi^{*}(\alpha(v))}.
$$ 
\end{lemma}

In order to study the existence of the global (universal) attractor, 
we introduce the following definitions.

\begin{definition}[See \cite{Temam}]
Let us consider $\displaystyle{\mathcal{B}\subset F}$ 
and $\displaystyle{\mathcal{U}}$ an open bounded set such that 
$\displaystyle{\mathcal{U}\subset\mathcal{B}}$. Then  
$\displaystyle{\mathcal{B}}$ is an absorbing set 
in $\displaystyle{\mathcal{U}}$ if the orbit of each bounded 
set of $\displaystyle{\mathcal{U}}$ enters into $\displaystyle{\mathcal{B}}$ 
after a given period of time (which may depend on the set):
$$
\begin{array}{ccc}\displaystyle{
\forall \mathcal{B}_{0} \subset \mathcal{U}, 
\quad \mathcal{B}_{0}} \text { bounded, } \\
\displaystyle{\exists s_{0}\left(\mathcal{B}_{0}\right)} 
\text { such that }\displaystyle{ S(s) \mathcal{B}_{0} 
\subset \mathcal{B}, 
\quad \forall s \geq s_{0}\left(\mathcal{B}_{0}\right)}.
\end{array} 
$$
\end{definition}

\begin{definition}[See \cite{Temam}]
\label{L999}
The set $A\subset F$ is said to be an universal attractor for the semigroup
$\displaystyle{\left(S(s)\right)_{s\ge 0}}$, if the following conditions hold:
\begin{itemize}
\item[(1)] $A\subset F$ is a nonempty invariant compact set,
	
\item[(2)] 	the set  $A\subset F$ attracts any bounded set
$\mathcal{B}\subset F$, that is,
$$
\displaystyle{	dist\left(S(s) \mathcal{B}, A\right) \rightarrow 0 }
\text { as } \displaystyle{s \rightarrow+\infty,} \\
\mbox{ such that } dist(D, B)
= {\operatorname{\sup}_{a \in D}}~{\operatorname{\inf}_{b \in B}}~\|a-b\|_F.
$$
\end{itemize}	
\end{definition}

% -----------------------------------------

\section{Regularized problems} 
\label{section3}

In this section, we first present our approximation scheme. 
Then we proceed to prove the existence of a weak solution 
to our regularized problem. To design our regularized scheme, 
we consider
\begin{equation} 
\label{ke1}
\begin{array}{ccccc}
\displaystyle\alpha_r  \mbox{ is of class }\displaystyle{ \mathcal{C}^1(\mathbb{R})} 
\mbox{ where }\displaystyle{ 0<\lambda<\alpha_r^\prime,} ~\\
\displaystyle{\alpha_r (0)=0, ~\alpha_{r} \longrightarrow \alpha}\mbox{ in }
\displaystyle{ \mathcal{C}_{loc}(\mathbb{R})} \mbox{ and }
\displaystyle{\left|\alpha_{r}\right| \leq|\alpha| },\\
\displaystyle{f_r} \mbox{ is of class }
\displaystyle \mathcal{C}^{\infty}(\mathbb{R}),\\ 
f_r \rightarrow f,\mbox{in } L^1(Q) \mbox{ and a.e in }Q, \\ 
\displaystyle{ f_r} \mbox{ satisfies  }\left(H3\right).
\end{array}
\end{equation}
The initial condition is regularized as in the proof 
of \cite[Proposition 3, p.~761]{filo1992global}, that is,
\begin{equation}
\label{w}
\begin{array}{cccc}
\displaystyle{v_{r, 0} \in \mathcal{C}^{\infty}_c(\Omega)} 
\mbox{  such that }\displaystyle{ v_{r,0 } \rightarrow v_{0} }
\mbox{ in }\displaystyle{L^{\infty}(\Omega),
~\left\|v_{r,0 }\right\|_{L^{\infty}(\Omega)} 
\leq\left\|v_{0}\right\|_{L^{\infty}(\Omega)}+1.}
\end{array}
\end{equation}
Our regularized problems are then given by
\begin{equation}
\label{P_r}
\quad \left\lbrace 
\begin{array}{lccccccc}	
\displaystyle \frac{\partial \alpha_r (v_r)}{\partial s}-\Delta_m^r v
=\kappa \frac{f_r(v_r)}{(\int_{\Omega}f_r(v_r)dx)^{2}},
&\mbox{in}& \quad Q=\Omega\times[0,M],\\
\displaystyle \alpha_r(v_{r,x}(0))=\alpha_r(v_{r,0}),
&\mbox{in}& \quad \Omega,\\
v_r= 0, &\mbox{on}&\quad \Gamma\times]0,M[, 
\end{array}
\right.  
\end{equation}
where $\displaystyle{\Delta_m^r v
=\operatorname{div}\left(\left(\mid \nabla v\mid^{2}
+r\right)^{\dfrac{m-2}{2}}\nabla v\right)}$, 
$\displaystyle{m\geq 2}$.

\begin{theorem}
\label{qsdlm}
Assume that hypotheses $(H1)$--$(H3)$ hold. Then there exists a 
solution to problem (\ref{P_r}).
\end{theorem}

The following lemma plays a key role in the proof of Theorem~\ref{qsdlm}.
\begin{lemma}
\label{A3}
For all $r>0$, we have 
$$
\displaystyle{\parallel v_r\parallel_{L^\infty(Q)}
\leq C(M,\parallel v_0\parallel_{L^\infty(\Omega)}),} 
$$
where $ \displaystyle{C(M,\parallel v_0\parallel_{L^\infty(\Omega)})}$ 
is a positive constant.
\end{lemma}

\begin{proof}
Multiplying the first equation of  problem (\ref{P_r}) 
by $\displaystyle{\left[\left( \alpha_r(v_r)
-\alpha_r(s_0)  \right)^+\right]^{p+1}}$ 
($s_0$ is a positive constant where $\mid v_r\mid >s_0$) 
and  integrating over $\displaystyle{\Omega}$, we get
$$	
\begin{aligned}
&\displaystyle{\int_\Omega \frac{\partial \alpha_r (v_r)}{\partial s}
\left[\left( \alpha_r(v_r)-\alpha_r(s_0) \right)^+\right]^{p+1}
-\int_\Omega\Delta_m^r v_r  \left[\left( \alpha_r(v_r)
-\alpha_r(s_0)  \right)^+\right]^{p+1}}\\
&\quad \displaystyle{=\int_\Omega \frac{\kappa \cdot f_r(v_r)}{
\left(\int_{\Omega}f_r(v_r)dx\right)^{2}}  
\left[\left( \alpha_r(v_r)-\alpha_r(s_0)  \right)^+\right]^{p+1}.}
\end{aligned}
$$
So, we have
\begin{multline*}
\displaystyle\dfrac{1}{p+2} 
\int_\Omega 	\frac{\partial }{\partial s} 
\left[\left( \alpha_r(v_r)-\alpha_r(s_0)  \right)^+\right]^{p+2}
=\int_\Omega\Delta_m^r v_r  \left[\left( 
\alpha_r(v_r)-\alpha_r(s_0)  \right)^+\right]^{p+1}\\
+\int_\Omega \frac{\kappa
\cdot f_r(v_r)}{\left(\int_{\Omega}f_r(v_r)dx\right)^{2}}  
\left[\left( \alpha_r(v_r)-\alpha_r(s_0)  \right)^+\right]^{p+1}.
\end{multline*}
Then,
\begin{multline}
\displaystyle \dfrac{1}{p+2} 
\frac{\partial }{\partial s}\int_\Omega
\left[\left( \alpha_r(v_r)-\alpha_r(s_0)\right)^+\right]^{p+2}
=\int_\Omega\Delta_m^r v_r  
\left[\left( \alpha_r(v_r)-\alpha_r(s_0)\right)^+\right]^{p+1}\\
+\int_\Omega \frac{\kappa\cdot f_r(v_r)}{
\left(\int_{\Omega}f_r(v_r)dx\right)^{2}}  
\left[\left( \alpha_r(v_r)-\alpha_r(s_0)  
\right)^+\right]^{p+1}.	
\end{multline}
On the other hand, we have
$$
\begin{aligned}
&\displaystyle{\int_\Omega\Delta_m^r v_r  
\left[\left( \alpha_r(v_r)-\alpha_r(s_0)  \right)^+\right]^{p+1} }\\
&=\displaystyle{\int_\Omega\operatorname{div}\left(
\left(\mid\nabla v_r\mid^{2}+r\right)^{\dfrac{m-2}{2}}
\nabla v_r\right)\left[\left( \alpha_r(v_r)-\alpha_r(s_0)\right)^+\right]^{p+1}}\\
&=\displaystyle{-(p+1)\int_\Omega\left(
\left(\mid\nabla v_r\mid^{2}+r\right)^{\dfrac{m-2}{2}}
\mid\nabla v_r\mid^2\right) \alpha' _r(v_r)  
\left[\left( \alpha_r(v_r)-\alpha_r(s_0)  \right)^+\right]^{p}}\\
&\quad +\displaystyle{\int_{\partial \Omega}\left(
\left(\mid\nabla v_r\mid^{2}+r\right)^{\dfrac{m-2}{2}}
\frac{\partial  v_r}{\partial \nu}\right)\left[\left( 
\alpha_r(v_r)-\alpha_r(s_0)  \right)^+\right]^{p+1}}.
\end{aligned}
$$
Since $\displaystyle{ 
\left(\mid\nabla v_r\mid^{2}
+r\right)^{\dfrac{m-2}{2}}
\mid\nabla v_r\mid^2\geq0}$ 
and $\alpha' _r>0$, we get
\begin{equation}
\begin{aligned}
&\displaystyle{\dfrac{1}{p+2} \frac{\partial }{\partial s}
\int_\Omega 	 \left[\left( \alpha_r(v_r)
-\alpha_r(s_0)  \right)^+\right]^{p+2}}\\
&\leq \displaystyle{ \int_{\partial \Omega}
\left((\mid\nabla v_r\mid^{2}+r)^{\dfrac{m-2}{2}}
\frac{\partial  v_r}{\partial \nu}\right)\left[
\left( \alpha_r(v_r)-\alpha_r(s_0)\right)^+\right]^{p+1} }\\
&\quad +\displaystyle{ \int_\Omega\frac{\kappa\cdot f_r(v)}{
\left(\int_{\Omega}f_r(v)dx\right)^{2}}  
\left[\left( \alpha_r(v_r)-\alpha_r(s_0)  
\right)^+\right]^{p+1}}.	
\end{aligned}
\end{equation}
By using $(H3)$, we have 
$$
\begin{aligned}
&\displaystyle\int_\Omega \frac{\kappa
\cdot f_r(v_r)}{(\int_{\Omega}f_r(v_r)dx)^{2}}  
\left[\left( \alpha_r(v_r)-\alpha_r(s_0)  
\right)^+\right]^{p+1}\\
&\qquad \leq \dfrac{\kappa}{\left(\sigma\cdot meas(\Omega)\right)^2}
\int_\Omega f_r(v_r)\left[\left( \alpha_r(v_r)
-\alpha_r(s_0)  \right)^+\right]^{p+1}.
\end{aligned}
$$
Since $f_r $ satisfies (H3), it yields
\begin{equation*}
\begin{split}
f_r\left(v_r(x,t)\right)
&=f_r\left(v_r(x,t)\right)\chi_{\lbrace v_r(x,t) 
\in supp(f)\rbrace}
+f_r\left(v_r(x,t)\right)\chi_{\lbrace v_r(x,t)
\notin supp(f)\rbrace}\\
&\leq f_r\left(v_r(x,t)\right)
\chi_{\lbrace v_r(x,t) \in supp(f)\rbrace}.
\end{split}
\end{equation*}
If  $v_r(x,t) \in supp(f)$, then it follows that 
$\left(v_r(x,t)\right)_r$ is bounded. Thus, 
there exists a positive constant $C_0$ such that 
$$
\displaystyle{\int_\Omega f_r(v_r)\left[\left( 
\alpha_r(v_r)-\alpha_r(s_0)  \right)^+\right]^{p+1}
\leq C_0 \int_\Omega \left[\left( \alpha_r(v_r)
-\alpha_r(s_0)  \right)^+\right]^{p+1}.}
$$
Keeping that in mind, we have 
for a positive constant $C_1$ that
\begin{equation}
\label{A2}
\displaystyle{
\dfrac{1}{p+2} \frac{\partial }{\partial s}
\int_\Omega 	 \left[\left( \alpha_r(v_r)-\alpha_r(s_0)  
\right)^+\right]^{p+2}\leq C_1\int_\Omega \left[
\left( \alpha(v_r)-\alpha_r(s_0)  \right)^+\right]^{p+1}}.
\end{equation}
From H\"{o}lder's inequality, there exists 
positive constants $C_j$, $j= 2,3,4$, such that
$$
\begin{aligned}
\displaystyle{
\int_\Omega \left[\left(\alpha_r(v_r)-\alpha_r(s_0)\right)^+\right]^{p+1}}
&\leq\displaystyle{\left(meas(\Omega)\right)^{\dfrac{1}{p+1}}
\cdot \left(\int_{\Omega}\left[(\alpha_{r}(v_r)
-\alpha_{r}(s_0))^+\right]^{{p+2}} \right)^{\dfrac{p+1}{p+2}}}\\
&\leq \displaystyle{C_2\left[  z_p(s)  \right]^{p+1}},
\end{aligned}
$$ 
where $\displaystyle{z_p(s)
:=\parallel (\alpha_r(v_r)-\alpha_r(s_0))^+\parallel_{L^{p+2}(\Omega)}}$.
In view of (\ref{A2}), we have
$$
\displaystyle{\dfrac{1}{p+2} \frac{\partial }{\partial s}
\int_\Omega \left[\left( \alpha_r(v_r)-\alpha_r(s_0)  
\right)^+\right]^{p+2}\leq C_3 \left[z_p(s) \right]^{p+1}}.
$$
Then, 
\begin{equation}
\displaystyle{
\dfrac{1}{p+2}\frac{\partial }{\partial s} 
\left[z_p(s) \right]^{p+2}\leq C_3 \left[z_p(s) \right]^{p+1}},
\end{equation}
and hence
$$
\displaystyle{\frac{\partial }{\partial s} \left[z_p(s) \right]
\leq C_3},
$$
from which it follows that
$$ 
\displaystyle{\left[z_p(s) -z_p(0)\right]\leq C_3 M,}
$$ 
which implies 
$$
\displaystyle{z_p(s)\leq z_p(0)+C_3M.}
$$
Letting $p$ go to infinity, we obtain that
\begin{equation}	
\label{H2}
\displaystyle{
\parallel\left( \alpha_r(v_r)-\alpha_r(s_0)  
\right)^+\parallel_{L^\infty(\Omega)}\leq C_4.} 
\end{equation}
Now, let $u_r=-v_r$, and consider the following problem:
\begin{equation}
\label{P_r1}
\quad \left\lbrace 
\begin{array}{lcccccc}
\displaystyle \frac{\partial \tilde{\alpha}_r (u_r)}{\partial s}
-\Delta_m^r u_r=\kappa \frac{\tilde{f}_r(u_r)}{
\left(\int_{\Omega}\tilde{f}_r(v_r)dx\right)^{2}}
=:\tilde{g}(u_r)~~~~~~&\mbox{in}&~~~~Q,\\
\displaystyle \tilde{\alpha}_{r} (u_{x,r}(0))
=\tilde{\alpha}_{r}(u_0) &\mbox{in}&\quad \Omega,\\
\displaystyle u_r= 0
&\mbox{on}&\quad \Gamma\times]0,M[,
\end{array}
\right.  
\end{equation}
where $\displaystyle{\tilde{\alpha}_r(\tau)
=-\alpha_r(-\tau)}$, $\displaystyle{\tilde{g}_r(\tau)
=-g_r(-\tau)}$ and 
$\displaystyle{\tilde{f}_r(\tau)=-f_r(-\tau)}$.  
Those functions satisfy the same conditions verified  
by $\alpha$,  $g$ and $f$, respectively. The same reasoning 
done to get (\ref{H2}), shows that 
\begin{equation}
\label{H3}
\displaystyle{
\parallel\left( \tilde{\alpha}_r(u_r)-\tilde{\alpha}_r(s_0)  
\right)^+\parallel_{L^\infty(\Omega)}\leq C_5,}
\end{equation}
which is equivalent to 
$$
\displaystyle{\parallel\left( -\alpha_r(-v_r(s))
+\alpha_r(-s_0)  
\right)^+\parallel_{L^\infty(\Omega)}\leq C_5. }
$$
From (\ref{H2}) and (\ref{H3}), we deduce that there 
exists a positive constant $C$ such that
$$
\displaystyle{
\parallel v_r (s) \parallel_{L^\infty(\Omega)}
\leq  C(M,\parallel v_0\parallel_{L^\infty(\Omega)}), 
\quad \text{ for all } s\in[0,M].}
$$ 
The lemma is proved.
\end{proof}

\begin{proof}[Proof of Theorem~\ref{qsdlm}]
From Lemma~\ref{A3} and hypotheses $(H1)$--$(H3)$, 
we conclude, from the classical results of Ladyzenskaya 
(see \cite[pp.~457--459]{ladyzhenskaya1967linear}),  
with the existence of a classical solution 
to the regularized problem (\ref{P_r}).
\end{proof}

% -----------------------------------------

\section{Existence of a weak solution} 
\label{section4}

\begin{definition}
We say that $\displaystyle{ v\in  L^\infty(Q) 
\cap L^{m}\left(0,M;W^{1,m}(\Omega)\right) 
\cap L^{\infty}\left(t,M;W^{1,m}(\Omega)\right)}$, $t>0$,  
is a bounded weak solution of problem  (\ref{(P)}), 
if it satisfies the following identity:
\begin{equation}
\displaystyle{\int_{0}^{M}\left\langle	\frac{\partial \alpha (v)}{\partial s},
u\right\rangle-\int_{Q}\mid \nabla v\mid^{m-2} \nabla  v \nabla u
=\kappa\int_{Q} \frac{f(v)}{(\int_{\Omega}f(v)dx)^{2}}u},
\end{equation}
for all $\displaystyle{u\in \left( L^{m}\left(0,M;W^{1,m}(\Omega)\right) 
\cap L^\infty(Q)\right)}$. Furthermore, if we have 
$$
\displaystyle{u\in  \left(W^{1,1}\left( 0,M;L^1(\Omega)\right) 
\cap L^{m}\left(0,M;W^{1,m}(\Omega)\right)  \right)}
$$
with $\displaystyle{u(\cdot,M)=0}$, then 
$$
\displaystyle{\int_{0}^{M}\left\langle
\frac{\partial \alpha (v)}{\partial s},
u\right\rangle=-\int_{0}^{M}\int_{\Omega}
\left[\alpha(v)-\alpha(v_0)\right]\partial_s u},
$$
where the duality product is defined by 
$\left\langle\cdot, \cdot\right\rangle=\left\langle
\cdot, \cdot\right\rangle_{W^{-1,m'}(\Omega),W^{1,m}(\Omega)}$.
\end{definition}

\begin{remark}
Since $\displaystyle{\alpha_r}$ is an increasing function and  
$\displaystyle{\mid \alpha_r \mid\leq \mid \alpha\mid}$, then, 
by using Lemma~\ref{A3}, we also have that $(\alpha_r(v_r ))_r$ 
is bounded.
\end{remark}

Our plan is to derive now enough \emph{a priori} 
estimates needed in the sequel.

\begin{lemma}
For all $r>0$, we have
\begin{equation}
\label{A7}
\displaystyle{	\left|\left|	v_r\right|\right|_{L^{m}
\left(0,M;W^{1,m}(\Omega)\right)}\leq C_6,}
\end{equation}   
where $C_6$ is a positive constant independent of $r$.
\end{lemma}

\begin{proof}
Multiplying the first equation of ${(\ref{P_r})}$ 
by $v_r$ and integrating, we get 
\begin{equation}
\label{A4}
\displaystyle{
\int_\Omega \frac{\partial \alpha_r(v_r)}{\partial s}v_r
-\int_\Omega\Delta_m^r v_r  v_r=\int_\Omega\kappa 
\frac{f_r(v_r)}{(\int_{\Omega}f_r(v_r)dx)^{2}}  v_r.}
\end{equation}
Applying (\ref{B2}), we obtain that
$$
\displaystyle{ \int_\Omega 
\frac{\partial \alpha_r(v_r)}{\partial s}v_r
=\int_\Omega \frac{\partial\left[\varPsi^*\left( 
\alpha_r(v_r)\right)\right]}{\partial s}.} 
$$
On another hand, by using  Green's formula, we get 
$$
\displaystyle{\int_\Omega\Delta_m^r v_r v_r
=-\int_\Omega \left(\mid\nabla v_r\mid^{2}
+r\right)^{\dfrac{m-2}{2}}
\nabla v_r\nabla v_r+\int_{\partial\Omega}\left(
\mid\nabla v_r\mid^{2}+r\right)
\frac{\partial  v_r}{\partial \nu}\cdot v_r.}
$$
Substituting into (\ref{A4}), we get 
$$
\begin{aligned}\displaystyle\int_\Omega 
\frac{\partial \alpha_r(v_r)}{\partial s}v_r
+\int_\Omega \left(\mid\nabla v_r\mid^{2}
+r\right)^{\dfrac{m-2}{2}}\mid\nabla v_r\mid^2
&=\int_\Omega \frac{\kappa 
\cdot f_r(v_r)}{(\int_{\Omega}f_r(v_r)dx)^{2}}v_r\\
&-\int_{\partial\Omega}\left(\mid\nabla v_r\mid^{m-2}+r\right)
\frac{\partial  v_r}{\partial \nu}\cdot v_r,
\end{aligned}
$$
$$
\begin{aligned}
\displaystyle
\int_\Omega \left(\mid\nabla v_r\mid^{2}
+r\right)^{\dfrac{m-2}{2}}\mid\nabla v_r\mid^2
&=\int_\Omega \frac{\kappa 
\cdot f_r(v_r)}{(\int_{\Omega}f_r(v_r)dx)^{2}}  
v_r-\int_\Omega \frac{\partial\left[
\varPsi^*\left( \alpha_r(v_r)\right)\right]}{\partial s}\\
&-\int_{\partial\Omega}\left(\mid\nabla v_r\mid^{m-2}
+r\right)\frac{\partial  v_r}{\partial \nu}\cdot v_r.	
\end{aligned}
$$
Then, using the boundary conditions, we have
\begin{equation}
\begin{aligned}
\displaystyle
\int_{0}^{M}\int_\Omega \left(\mid\nabla v_r\mid^{2}
+r\right)^{\dfrac{m-2}{2}}\mid\nabla v_r\mid^2
&=\int_{0}^{M}\int_\Omega \frac{\kappa 
\cdot f_r(v_r)}{(\int_{\Omega}f_r(v_r)dx)^{2}}  v_r\\
&-\int_{0}^{M}\int_\Omega \frac{\partial\left[
\varPsi^*\left( \alpha_r(v_r)\right)\right]}{\partial s}.
\end{aligned}
\end{equation}
From Remark~\ref{yy}, we know that 
$\left(\varPsi^*\left( \alpha_r(v_r) \right)\right)_r$ 
is bounded. With the aid of hypothesis $(H3)$ and Lemma~\ref{A3}, 
there exists a positive constant $C_7$ such that 
$$
\begin{aligned}
&\displaystyle\int_{0}^{M}\int_\Omega\kappa 
\frac{f_r(v_r)}{(\int_{\Omega}f_r(v_r)dx)^{2}}  v_r-\int_{0}^{M}
\int_\Omega \frac{\partial\left[\varPsi^*\left( 
\alpha_r(v_r)\right)\right]}{\partial s}\\
&\leq \int_{0}^{M}\int_\Omega\kappa \frac{f_r(v_r)}{(\int_{\Omega}
f_r(v_r)dx)^{2}}  v_r\\
&\quad -\int_{\Omega}\varPsi^*\left( \alpha_r(v_r(\cdot,M))\right) 
+  \int_{\Omega}\varPsi^*\left( \alpha_r(v_r(\cdot,0))\right)\\
&\leq \dfrac{\kappa}{\left(\sigma\cdot meas(\Omega)\right)^2}
\int_{0}^{M}\int_\Omega f_r(v_r)\cdot\mid v_r\mid\\
&\quad +2\cdot\max\left\{\left| \int_{\Omega}\varPsi^*\left( 
\alpha_r(v_r(\cdot,M))\right) \right|,\left|  
\int_{\Omega}\varPsi^*\left( \alpha_r(v_r(\cdot,0))\right) 
\right|\right\} 
\leq C_7.
\end{aligned}
$$
It yields that
$$ 
\displaystyle{\int_{0}^{M}\int_\Omega \left|\nabla v_r\right|^m 
\leq\int_{0}^{M}\int_\Omega \left(\mid\nabla v_r\mid^{2}
+r\right)^{\dfrac{m-2}{2}}\mid\nabla v_r\mid^2\leq C_7.}
$$
We deduce that $v_r\in L^{m}\left(0,M;W^{1,m}(\Omega)\right)$.
\end{proof}

\begin{remark}
\label{LS}
Inequality (\ref{A7}), combined with Young's inequality,  
imply that 
$$
\left(\left(\mid\nabla v_r\mid^{2}
+r\right)^{\dfrac{m-2}{2}}\nabla v_r\right)_r 
$$ 
is bounded in $L^{m'}\left(0,M;W^{1,m'}(\Omega)\right)$.
\end{remark}

A further upper bound for $v_r$ is established in the following lemma.

\begin{lemma}
\label{L52}
For all $r,s>0$, there exist positive constants $C(t)$,
$C(t,M)$, and $C_1(t,M)$, such that the following 
inequalities hold:
\begin{equation}
\label{L4}
\displaystyle{
\left|\left|v_r\left(s\right)\right|\right|_{W^{1,m}(\Omega)}\leq C(t),
\quad \mbox{ for all } s\geq t,}
\end{equation}
\begin{equation}
\displaystyle{
\int_{t}^{M}\int_\Omega \alpha_r'(v_r)\left(
\dfrac{\partial v_r }{\partial s}\right)^2\leq C(t,M),}
\end{equation}
\begin{equation}
\label{A8}
\displaystyle{
\int_{t}^{M}\int_\Omega \left(
\dfrac{\partial \alpha_r(v_r)}{\partial s} \right)^2\leq C_1(t,M)}.
\end{equation}
\end{lemma}

\begin{proof}
Multiplying the first equation of problem ${(\ref{P_r})}$ 
by $\displaystyle{\dfrac{\partial v_r}{\partial s}}$, 
and integrating, we obtain that
\begin{equation}
\label{A5}
\displaystyle{
\int_\Omega \frac{\partial \alpha_r(v_r)}{\partial s}
\dfrac{\partial v_r}{\partial s}-\int_\Omega\Delta_m^r v_r  
\dfrac{\partial v_r}{\partial s}=\int_\Omega\kappa 
\frac{f_r(v_r)}{(\int_{\Omega}f_r(v_r)dx)^{2}}  
\dfrac{\partial v_r}{\partial s}.}
\end{equation}
Since
$$	
\displaystyle{\int_\Omega \frac{\partial \alpha_r(v_r)}{\partial s}
\dfrac{\partial v_r}{\partial s}
=\int_\Omega \alpha_r'(v_r)\left( 
\dfrac{\partial v_r}{\partial s}\right)^2},
$$
the equality (\ref{A5}) becomes
$$
\displaystyle{
\int_\Omega \alpha_r'(v_r)\left( 
\dfrac{\partial v_r}{\partial s}\right)^2
-\int_\Omega\Delta_m^r v_r  
\dfrac{\partial v_r}{\partial s}
=\int_\Omega\kappa \frac{f_r(v_r)}{(\int_{\Omega}f_r(v_r)dx)^{2}}  
\dfrac{\partial v_r}{\partial s}.}
$$
By applying Green's formula, we get 
\begin{equation}
\displaystyle{
\int_\Omega \alpha_r'(v_r)\left( \dfrac{\partial v_r}{\partial s}\right)^2
+\dfrac{1}{m}\dfrac{\partial }{\partial s}\int_\Omega \left(
\mid\nabla v_r\mid^{2}+r\right)^{\dfrac{m}{2}} 
=\int_\Omega\kappa \frac{f_r(v_r)}{(\int_{\Omega}f_r(v_r)dx)^{2}}  
\dfrac{\partial v_r}{\partial s}}.
\end{equation}
Let $\displaystyle{ G_r(v_r):=\int_{0}^{v_r} g_r(s)ds }$ 
and $\displaystyle{g_r(s):=\frac{f_r(s)}{(\int_{\Omega}f_r(s)dx)^{2}}}$.
By using the boundedness of $\displaystyle{v_r}$ and (\ref{ke1}), 
we have  $\displaystyle{\dfrac{\partial G_r(v_r)}{\partial s}\leq C_8}$.
Then, it yields that 
$$
\displaystyle{\int_{\Omega}g_r(v_r)
\dfrac{\partial v_r}{\partial s}
\leq C_8\cdot meas(\Omega).}
$$
With this in mind, we derive 
\begin{equation}
\label{A6}
\displaystyle{
\int_\Omega \alpha_r'(v_r)\left( 
\dfrac{\partial v_r}{\partial s}\right)^2+\dfrac{1}{m}
\dfrac{\partial }{\partial s}\int_\Omega \left(\mid\nabla 
v_r\mid^{2}+r\right)^{\dfrac{m}{2}} \leq C_9.}
\end{equation}
Then,
\begin{equation}
\label{kw}
\displaystyle{\dfrac{1}{m}\dfrac{\partial }{\partial s}
\int_\Omega \left(\mid\nabla v_r\mid^{2}+r\right)^{\dfrac{m}{2}} 
\leq C_9}
\end{equation}
and, by using Gronwall's Lemma~\ref{Lk}, we get 
\begin{equation}
\label{4dd4}
\displaystyle{\int_\Omega \left|\nabla v_r\right|^m 
\leq\dfrac{1}{m}\int_\Omega \left(\mid\nabla 
v_r\mid^{2}+r\right)^{\dfrac{m}{2}} \leq C_{10}}.
\end{equation}
According to Poincar\'{e}'s inequality, it follows that
$$
\displaystyle{\left|\left|v_r\left(s\right)\right|
\right|_{W^{1,m}(\Omega)}\leq C(t),
\quad \mbox{ for all }~~s\geq t}.
$$
This, combined with inequality (\ref{A6}), yields
\begin{multline}
\label{fg4}
\displaystyle\int_{t}^{M}\int_\Omega \alpha_r'(v_r)\left( 
\dfrac{\partial v_r}{\partial s}\right)^2 
+\dfrac{1}{m}\int_\Omega \left(\mid\nabla v_r(\cdot,M)\mid^{2}
+r\right)^{\dfrac{m}{2}}\\
\leq \dfrac{1}{m}
\int_\Omega \left(\mid\nabla 
v_r(\cdot,t)\mid^{2}+r\right)^{\dfrac{m}{2}}
+C_9\left(M-t \right).		
\end{multline}
Now, add (\ref{4dd4}) to (\ref{fg4}), to obtain
$$
\displaystyle{\int_{t}^{M}\int_\Omega \alpha_r'(v_r)
\left( \dfrac{\partial v_r}{\partial s}\right)^2 
+\dfrac{1}{m}\int_\Omega \left(\mid\nabla 
v_r(\cdot,M)\mid^{2}+r\right)^{\dfrac{m}{2}}\leq C(t,M).}
$$
As a consequence, we have
$$
\displaystyle{\int_{t}^{M}\int_\Omega \alpha_r'(v_r)
\left( \dfrac{\partial v_r}{\partial s}\right)^2\leq C(t,M).}
$$
Since $\alpha$ is a locally Lipschitzian function, 
then there exists a positive constant $L$ such that 
$\displaystyle{\alpha_r'\leq L}$.  Hence, we get 
$$	
\displaystyle{\int_{t}^{M}\int_\Omega 
\left(\dfrac{\partial \alpha_r(v_r)}{\partial s} \right)^2
\leq L \int_{t}^{M}\int_\Omega \alpha_r'(v_r)\left( 
\dfrac{\partial v_r}{\partial s}\right)^2\leq  C_1(t,M).}
$$
The proof is complete.
\end{proof}

\begin{theorem}
\label{qsdlm:ws}
Assume that hypotheses $(H1)$--$(H3)$ hold. Then there exists a 
weak bounded solution to problem (\ref{P_r}).
\end{theorem}

\begin{proof}
To achieve the proof of Theorem~\ref{qsdlm:ws}, 
we need to pass to the limit in problem $\eqref{P_r}$.
By virtue of Lemma  ${\ref{A3}}$, there exists 
a subsequence, still denoted  
$\displaystyle{\left(v_r\right)_{r}},$  such that 
$$
\displaystyle{v_r\longrightarrow v~~\mbox{weakly star in}~~L^\infty(Q)}.
$$
Note from estimate ${(\ref{A7})}$ that 
$$ 
\displaystyle{v_r\longrightarrow 
v~~\mbox{weakly in~}~L^{m}\left(0,M;W^{1,m}(\Omega)\right)}.
$$
Since $\left(v_r\right)_{r}$ is bounded in 
$\displaystyle{L^{\infty}\left(t,M;W^{1,m}(\Omega)\right)}$, 
then 
$$ 
\displaystyle{v_r\longrightarrow v~~\mbox{weakly 
star in }~L^{\infty}\left(t,M;W^{1,m}_0(\Omega)\right)}.
$$
Under the hypotheses of $f_r$, we have $f_r\longrightarrow f$ a.e. 
This, together with Vitali's theorem (see \cite{reynolds2004vitali}), 
implies the convergence to $f(v)$ in $L^1(Q)$.
Applying Green's formula,
$$
\displaystyle{\left|\int_{0}^{M}\int_\Omega\Delta_m^r v_r u\right|}
\le\displaystyle{\left|\int_\Omega \left(\mid\nabla v_r\mid^{2}
+r\right)^{\dfrac{m-2}{2}}\nabla v_r\nabla u\right| },~\mbox{ for } 
\displaystyle{u\in L^{m}\left(0,M;W^{1,m}_0(\Omega)\right)}.
$$
By using Remark~\ref{LS}, the right-hand side of this inequality is bounded. 
Then there exists 
$\displaystyle{\vartheta\in L^{m'}\left(0,M;W^{-1,m'}(\Omega)\right)}$ 
such that 
$$
\displaystyle{\Delta_m^r v_r\longrightarrow \vartheta}~~\mbox{weakly in }
\displaystyle{ ~~L^{m'}\left(0,M;W^{-1,m'}(\Omega)\right)}.
$$
A classical argument (see \cite{blanchard1988study}), 
asserts that $\displaystyle{\vartheta =\Delta_m v}$.

Combining (\ref{L4}) and the smoothness of function $\alpha_r$, 
yields the boundedness of the sequence 
$\displaystyle{\left(\alpha_r(v_r)\right)_{r}}$ 
in $\displaystyle{ L^{\infty}\left(t,M;W^{1,m}(\Omega)\right)}$. 
On the other hand, by using ${(\ref{A8})}$, we deduce that 
$\displaystyle{\left(\dfrac{\partial \alpha_r(v_r)}{\partial s}\right)_{r}}$ 
is bounded in $L^{2}\left(t,M;L^{2} (\Omega)\right)$, for all $t>0$. 
Aubin's lemma (see \cite{Jacque86}) allows us to claim that  
$\displaystyle{\left(\alpha_r(v_r)\right)_{r}}$ is relatively compact 
in $\displaystyle{C\left(]0,M [;L^1(\Omega) \right)}$. Therefore, 
$\displaystyle{\alpha_r(v_r)\longrightarrow \delta}$ strongly in 
$\displaystyle{C\left(]0,M [;L^1(\Omega) \right)}$. Hence, in an entirely similar manner as in 
\cite[p.~1048]{blanchard1988study}, it can be handled that $\displaystyle{\delta=\alpha(v)}$.
For the continuous of the solution at point $s=0$, we proceed as in \cite{andreu1998attractor}.  
From Lemma~\ref{L52}, we deduce that $\displaystyle{\alpha_{r}(v_{r} )\longrightarrow \alpha(v)}$ 
strongly in $\displaystyle{C\left([0,M];L^1(\Omega) \right)}$. 

Let us consider $\displaystyle{v_0\in L^{\infty}(\Omega)}$ and take a smooth sequence 
$\left(v_{r,0}\right)$ satisfying (\ref{w}). Hence, $\left(v_{r,0}\right)$  
is bounded and convergent to $v_{0}$ in $L^1(\Omega)$.
Then, thanks to the dominate convergence theorem, we have 
$\displaystyle{\alpha\left( v_{r,0}\right)\longrightarrow \alpha\left( v_{0}\right)}$ 
in  $\displaystyle{L^{1}(\Omega)}$. Now, we deal with initial data  
$\displaystyle{v_0\in C^1(\bar{\Omega})}$. Choosing the sequence 
$\displaystyle{\left( v_{r,0}\right)}$ bounded in the space 
$\displaystyle{W^{1,m}(\Omega)}$ and verifying  hypothesis ${(\ref{w})}$, 
the corresponding  $\displaystyle{\alpha(v_{r} )}$ are continuous at $s=0$. 
Furthermore, we have
\begin{equation}
\begin{split}
\|\alpha(v(s))-\alpha(v(0))\|_{L^{1}(\Omega)}
&\leq \left\|\alpha(v(s))-\alpha\left(v_{r}(s)\right)\right\|_{L^{1}(\Omega)}
+\displaystyle{\left\|\alpha\left(v_{r}(s)\right)
-\alpha\left(v_{r,0 }\right)\right\|_{L^{1}(\Omega)}}\\
&\quad +\displaystyle{\left\|\alpha\left(v_{r,0 }\right)
-\alpha\left(v_{0}\right)\right\|_{L^{1}(\Omega)}}.
\end{split}
\end{equation}
In view of Lemma~\ref{L53}, we have
\begin{equation}
\label{key19}
\begin{split}
\|\alpha&(v(s))-\alpha(v(0))\|_{L^{1}(\Omega)}
\displaystyle{\leq e^{K s}\left\|\alpha\left( v_{0}\right)
-\alpha\left(v_{r,0}\right)\right\|_{L^{1}(\Omega)}}\\
&\quad +\displaystyle{\left\|\alpha\left(v_{r}(s)\right)
- \alpha \left(v_{r,0 }\right)\right\|_{L^{1}(\Omega)}}
+\displaystyle{\left\|\alpha\left(v_{r,0 }\right)
-\alpha\left(v_{0}\right)\right\|_{L^{1}(\Omega)}}.
\end{split}
\end{equation}
As $s$ goes to $0$ of (\ref{key19}), all terms  
of the right hand side of $ {(\ref{key19})}$ tend to $0$. 
Then, we deduce that 
$\displaystyle{\alpha\left(v\right)\in C\left([0,M];L^1(\Omega) \right)}$.
Finally, letting $\displaystyle{r\longrightarrow 0}$ in $ {(\ref{P_r})}$, 
we obtain the existence of a weak bounded solution.
\end{proof}

% -----------------------------------------

\section{Uniqueness of solution}
\label{section55}

To prove the uniqueness of the solution, we need to impose
some further hypothesis. We assume that there exists 
a positive constant $L_2$ such that
\begin{equation}
\label{B1}
\displaystyle{	\mid f(u)-f(v)\mid
\leq L_2\mid\alpha(u)-\alpha(v) \mid.}
\end{equation}

\begin{lemma}
\label{L53}
Let $v$ and $u$ be two solutions of problem (\ref{(P)}) 
with initial data $v_{0}$ and $u_{0}$, respectively. 
Then, the following inequality holds:
\begin{equation}
\displaystyle{
\|\alpha(v(s))-\alpha(u(s))\|_{L^{1}(\Omega)} 
\leq e^{K s}\left\|\alpha\left( v_{0}\right)
-\alpha\left(u_{0}\right)\right\|_{L^{1}(\Omega)}},
\end{equation}
where $K$ is a positive constant.
\end{lemma}

\begin{proof} 
The proof is similar to the one in \cite{diaz1994nonlinear}.
\end{proof}

For the proof of our next result, 
we need the following lemma.

\begin{lemma}[Tartar's inequality \cite{simon1981regularite}]
\label{53}
If $a,b\in\mathbb{R}^N$, then 
\begin{equation}
\label{ii}
\displaystyle{	\left[|a|^{m-2} a-|b|^{m-2} b\right] 
\cdot(a-b) \geq C(m)
\left\{
\begin{array}{ll}
|a-b|^{m}, & \text { if } m \geq 2, \\
\frac{|a-b|^{2}}{(|a|+|b|)^{2-m}}, & \text { if } 1<m<2,
\end{array}\right.}
\end{equation}
for all $m>1$, where $C(m)=2^{2-m}$ 
when $m \geq 2$ and $C(m)=m-1$ when $1<m<2$.
\end{lemma}

\begin{lemma}
\label{533}
Let us consider two solutions $\displaystyle{v}$ and $\displaystyle{u}$ 
of problem (\ref{(P)}) with initial data $\displaystyle{v_{0}}$ 
and $\displaystyle{u_{0}}$, respectively, such that 
$\displaystyle{v_{0}= u_{0}}$. Then, 
$\displaystyle{v= u}$ in $Q$.
\end{lemma}

\begin{proof}
For a small positive $\displaystyle{\mu}$, let
$$
\displaystyle H_{\mu}(Y)
=\min\left\lbrace 1,\max\left\lbrace \dfrac{Y}{\mu},0\right\rbrace 
\right\rbrace, \mbox{ for all }  Y\in \mathbb{R}.
$$
We use $\displaystyle{H_{\mu}(v-u)}$ as a test function. Multiplying 
the first equation of problem (\ref{(P)}), corresponding to $u$ and $v$,  
by $\displaystyle{H_{\mu}(v-u)}$ and subtracting the two equations, 
we derive that 	
\begin{equation}
\label{A9}
\begin{array}{cccccc}
&\displaystyle	
\int_{0}^{s}\int_\Omega \frac{\partial }{\partial s} 
\left( \alpha(v)-\alpha(u) \right)H_{\mu}(v-u)
-\int_{0}^{s}\int_\Omega\left(\Delta_m v
-\Delta_m u\right) H_{\mu}(v-u) \\[0.3cm]
&=\displaystyle{\int_{0}^{s}\int_\Omega\kappa 
\frac{f(v)}{(\int_{\Omega}f(v)dx)^{2}}  H_{\mu}(v-u)}
-\displaystyle{\int_{0}^{s}\int_\Omega\kappa 
\frac{f(u)}{(\int_{\Omega}f(u)dx)^{2}}  H_{\mu}(v-u).}
\end{array}
\end{equation}
Using Green's formula and taking into account the boundary conditions,  
we obtain that
\begin{equation} 
\label{eky}
\displaystyle{\int_{0}^{s}\int_\Omega\left(\Delta_m v\right) H_{\mu}(v-u)
=-\int_{0}^{s}\int_\Omega \mid \nabla v\mid^{m-2} 
\nabla v\cdot\nabla(v-u)\cdot H_{\mu}'(v-u )}.
\end{equation}
We easily check that
\begin{equation} 
\label{eeky}
\displaystyle{\int_{0}^{s}\int_\Omega\left(\Delta_m u\right) H_{\mu}(v-u)
=-\int_{0}^{s}\int_\Omega \mid \nabla u\mid^{m-2} \nabla u
\cdot\nabla(v-u)\cdot H_{\mu}'(v-u ) }.
\end{equation}
From (\ref{eky}) and (\ref{eeky}), it follows that
$$
\begin{aligned}
&\displaystyle\int_{0}^{s}\int_\Omega\left(\Delta_m v
-\Delta_m u \right) \cdot H_{\mu}(v-u )\\
&=-\int_{0}^{s}\int_\Omega \left[\mid \nabla v\mid^{m-2} \nabla v
-\mid \nabla u \mid^{m-2} \nabla u \right]\nabla(v-u )
\cdot H_{\mu}'(v-u ).	
\end{aligned}
$$
By using Lemma~\ref{53}, it follows that
$$
\displaystyle{\int_{0}^{s}\int_\Omega\left(\Delta_m v
-\Delta_m u \right) \cdot H_{\mu}(v-u )\leq 0}.
$$
Hence,
\begin{equation}
\label{keee}
\begin{aligned}
\displaystyle\int_{0}^{s}\int_\Omega \frac{\partial }{\partial s} 
\left( \alpha(v)-\alpha(u) \right)H_{\mu}(v-u  )
&\leq \int_{0}^{s}\int_\Omega \frac{\partial }{\partial s} 
\left( \alpha(v)-\alpha(u) \right)H_{\mu}(v-u)\\
&-\int_{0}^{s}\int_\Omega\left(\Delta_m v-\Delta_m u\right) 
H_{\mu}(v-u).
\end{aligned}
\end{equation}
Recalling (\ref{A9}) and (\ref{keee}), we get
\begin{equation}
\label{wifg}
\begin{array}{cccccc}
\displaystyle{
\int_{0}^{s}\int_\Omega \frac{\partial }{\partial s} 
\left( \alpha(v)-\alpha(u) \right)\cdot H_{\mu}(v-u)}
&\displaystyle{\leq}
&\displaystyle{\int_{0}^{s}
\int_\Omega \gamma(x)\cdot  H_{\mu}(v-u )},
\end{array}
\end{equation}
where
$$
\displaystyle{\gamma(x) 
:=\kappa \dfrac{f(v)}{(\int_{\Omega}f(v)dx)^{2}}
-\kappa\dfrac{f(u)}{(\int_{\Omega}f(u)dx)^{2}}},
$$
$$
\begin{aligned}
\displaystyle\gamma(x)\cdot \chi_{\lbrace v-u >0\rbrace}
&=\kappa f(u)\dfrac{\int_{\Omega}\left[f(u)-f(v)\right]\mathrm{~d} x
\int_{\Omega}\left[ f(u)+f(v)   \right]\mathrm{~d}x}{
\left(\int_{\Omega}f(u)dx\right)^{2}\left(\int_{\Omega}f(v)dx\right)^{2}}
\cdot \chi_{\lbrace v-u >0\rbrace}\\
&+\kappa\dfrac{f(v)-f(u)}{\left(\int_{\Omega}f(v)dx\right)^{2}}
\cdot \chi_{\lbrace v-u >0\rbrace}.
\end{aligned}
$$
Adding this to (\ref{B1}), 
\begin{equation} 
\label{643197}
\begin{aligned}
\displaystyle\gamma(x)\cdot \chi_{\lbrace v-u >0\rbrace}
&\le \displaystyle{\kappa L_2 \dfrac{\int_{\Omega}\left( 
\alpha(v)-\alpha(u)\right) \mathrm{~d} x\left(   
\int_{\Omega}( f(v) + f(u)) \mathrm{~d} x\right)}{
\left(\int_{\Omega}f(u)dx\right)^{2}\left(\int_{\Omega}f(v)dx\right)^{2}}}
f(u)\cdot \chi_{\lbrace v-u >0\rbrace}\\
&+\displaystyle\kappa L_2 \dfrac{\left( \alpha(v)
-\alpha(u)\right)}{\left(\int_{\Omega}f(v)\mathrm{~d} x\right)^{2}}
\cdot \chi_{\lbrace v-u >0\rbrace}.
\end{aligned}		
\end{equation}
On the other hand, we have  
\begin{equation}
\label{yqqy}
\begin{array}{cccccccccccccc}
& \displaystyle{\kappa\cdot L_2 \int_{0}^{s}
\int_\Omega\dfrac{\left(\int_{\Omega}\left( \alpha(v)-\alpha(u)\right) 
\mathrm{~d} x\right)\left(\int_{\Omega}( f(v) + f(u)) \mathrm{~d} x\right) }{
\left(\int_{\Omega}f(u)dx\right)^{2}\left(\int_{\Omega}f(v)dx\right)^{2}}
f(u)\cdot \chi_{\lbrace v-u >0\rbrace}}\\
&\le \displaystyle{2 \kappa \cdot L_2\cdot meas(\Omega)
\cdot\underset{a \in \operatorname{supp}(f)}{\sup f(a)} 
\int_{0}^{s}\int_\Omega\dfrac{\left(\int_{\Omega}\left( 
\alpha(v)-\alpha(u)\right) \mathrm{~d} x\right)}{
\left(\int_{\Omega}f(u)dx\right)^{2}\left(\int_{\Omega}f(v)dx\right)^{2}}f(u)}\\
&\le \displaystyle{2 \kappa\cdot L_2\cdot meas(\Omega)
\cdot\left( \underset{	a \in \operatorname{supp}(f)}{\sup f(a)}\right)^2 
\int_{0}^{s}\int_\Omega\dfrac{\left(\int_{\Omega}\left( 
\alpha(v)-\alpha(u)\right) \mathrm{~d} x\right) }{
\left(\int_{\Omega}f(u)dx\right)^{2}
\left(\int_{\Omega}f(v)dx\right)^{2}}}.
\end{array} 
\end{equation}
Since 
$$
\begin{aligned}
\int_{\Omega}\left( \alpha(v)-\alpha(u)\right) \mathrm{~d} x
&=\int_{\Omega}\left( \alpha(v)-\alpha(u)\right)
\cdot \chi_{\lbrace v-u >0\rbrace} \mathrm{~d} x\\
&\quad + \int_{\Omega}\left( \alpha(v)-\alpha(u)\right)
\cdot \chi_{\lbrace v-u \le 0\rbrace} \mathrm{~d} x, 
\end{aligned}
$$ 
and $\alpha$ is an increasing function, we get that
\begin{equation}
\label{bgbh}
\int_{\Omega}\left( \alpha(v)-\alpha(u)\right) \mathrm{~d} x 
\le \int_{\Omega}\left( \alpha(v)-\alpha(u)\right)
\cdot \chi_{\lbrace v-u >0\rbrace} \mathrm{~d} x
\le \int_{\Omega}\left( \alpha(v)-\alpha(u)\right)^+\mathrm{~d} x.
\end{equation}
Keeping in mind (\ref{643197})--(\ref{bgbh})
and hypothesis $(H3)$ on $f$, it follows that
\begin{equation}	
\label{1ey25}
\begin{aligned}
&\displaystyle{	\int_{0}^{s} \int_{\Omega}\gamma(x)
\cdot \chi_{\lbrace v-u >0\rbrace}     \mathrm{~d} x \mathrm{~d} t} \\
&\leq\displaystyle{\dfrac{\kappa \cdot L_2}{(meas(\Omega)\sigma)^2}  
\int_{0}^{s} \int_{\Omega}\left( \alpha(v)-\alpha(u)\right)^+ 
\mathrm{~d} x \mathrm{~d} t}\\
&\quad + \displaystyle{\dfrac{2\kappa \cdot L_2\cdot meas(\Omega)}{
\left(meas(\Omega)\cdot\sigma\right)^4}
\cdot\left( \underset{	a\in \operatorname{supp}(f)}{\sup f(a)}\right)^2
\int_{0}^{s}\int_\Omega\left(\int_{\Omega}\left( 
\alpha(v)-\alpha(u)\right)^+ \mathrm{~d} x\right)}\\
&\leq \displaystyle{\left(\dfrac{\kappa L_2}{(meas(\Omega)\sigma)^2} 
+\dfrac{2\kappa \cdot L_2 (meas(\Omega))^2}{(meas(\Omega)\sigma)^4}
\left( \underset{a \in \operatorname{supp}(f)}{\sup f(a)}\right)^2\right)
\int_{0}^{s} \int_{\Omega}\left( \alpha(v)-\alpha(u)\right)^+ 
\mathrm{d} x \mathrm{d} t}.
\end{aligned}
\end{equation}
On the another hand, when we tend $\displaystyle{\mu}$ to zero, we get
$$
\displaystyle{\int_{0}^{s}\int_\Omega 
\frac{\partial }{\partial s} \left( \alpha(v)
-\alpha(u ) \right)H_{\mu}(v-u ) \longrightarrow 
\int_{0}^{s}\int_\Omega \frac{\partial }{\partial s} 
\left( \alpha(v)-\alpha(u  ) \right)
\cdot\chi_{\lbrace v-u >0\rbrace}}.
$$ 
We also have that
$$
\displaystyle{\int_{0}^{s}\int_\Omega \gamma(x)
\cdot H_{\mu}(v-u)\longrightarrow  \int_{0}^{s}
\int_\Omega \gamma(x)\cdot \chi_{\lbrace v-u >0\rbrace}.}
$$
This, combined with (\ref{wifg}) and (\ref{1ey25}),  
yields the existence of a positive constant $C_{11}$ such that
\begin{equation}
\displaystyle{
\int_{\Omega}(\alpha(v)-\alpha(u ))^+ \leq C_{11}
\cdot \int_{0}^{s} \int_{\Omega}(\alpha(v)-\alpha(u ))^+.}
\end{equation}
Applying the usual Gronwall's lemma, we get 
$\displaystyle{\alpha(v)\leq\alpha(u )}$.  Knowing that 
$\alpha$ is an increasing function, it follows, in particular, 
that $\displaystyle{\alpha(v)=\alpha(u )}$ 
in $\displaystyle{\lbrace v-u >0\rbrace}$. 
Keeping this and (\ref{ii})  in mind,  
we obtain that $\displaystyle{\nabla( v-u )=0}$   
in $\displaystyle{\lbrace\mu >v-u >0\rbrace}$. 
Hence, $\displaystyle{\max\lbrace 0, \min\lbrace v-u,\mu \rbrace\rbrace
=C_{12}}$, where $C_{12}$ is a positive constant. We deduce that 
$\displaystyle{v\leq u}$ in $Q$. Interchanging
the role of $v$ and $u$, the proof of uniqueness is finished.
\end{proof}

% -----------------------------------------

\section{Existence of an absorbing set and the universal attractor}
\label{section6}
 
In this section we prove the existence of an universal attractor 
by first proving the existence of an absorbing set. To this end, 
let us consider $\displaystyle{\left(S(s)\right)_{s\ge  0}}$ 
a continuous semigroup generated by problem (\ref{(P)}) such that
\begin{equation}
\label{L88}
\begin{array}{cccccc}\displaystyle{S(s):}
&\displaystyle{L^{\infty}(\Omega)}
&\displaystyle{\rightarrow L^{\infty}(\Omega)}\\
&\displaystyle{v_0}&\displaystyle{\rightarrow \alpha(v(s)),}
\end{array}
\end{equation}
where $v$ is the bounded weak solution of problem (\ref{(P)}). 
By using Theorem~\ref{qsdlm}, the map (\ref{L88}) is well defined. 
Now, let us formulate the second main result in this paper.

\begin{theorem}
\label{L5} 
For $m>2$,
$\displaystyle{\left(S(s)\right)_{s\ge  0}}$ 
possesses an universal attractor, which is bounded 
in $W^{1,m}_0(\Omega)$.
\end{theorem}

In order to prove Theorem~\ref{L5}, 
we first show the following result.

\begin{lemma}
Under assumptions $(H1)$--$(H3)$, there exists a positive constant 
$\displaystyle{\rho}$ such that 
$$
\displaystyle{\parallel v (s) \parallel_{L^\infty(\Omega)}
\leq \rho, \quad \mbox{for all }~s>0.}
$$ 
\end{lemma}

\begin{proof}
Multiplying the first equation of $\displaystyle{ {(\ref{(P)})}}$ by
$\displaystyle{\left|\alpha(v)\right|^p\alpha(v)}$, and integrating 
over $\Omega$, we obtain that
$$
\displaystyle{\int_{\Omega}\frac{\partial 
\alpha (v)}{\partial s}\left|\alpha(v)\right|^p\alpha(v)
-\int_{\Omega}\Delta_m v \cdot\left|\alpha(v)\right|^p\alpha(v)
=\kappa \int_{\Omega}\frac{f(v)}{(\int_{\Omega}f(v)dx)^{2}}
\left|\alpha(v)\right|^p\alpha(v).}
$$
Then, 
$$
\displaystyle{\dfrac{1}{p+2}\frac{\partial}{\partial s}
\int_{\Omega}\left|\alpha(v)\right|^{p+2}-\int_{\Omega}\Delta_m v
\cdot\left|\alpha(v)\right|^p\alpha(v)
=\kappa \int_{\Omega}\frac{f(v)}{(\int_{\Omega}f(v)dx)^{2}}
\left|\alpha(v)\right|^p\alpha(v).} 
$$
Applying Green's formula, and using the boundary conditions, we get
\begin{multline}
\label{L3}
\dfrac{1}{p+2}\frac{\partial }{\partial s}
\int_{\Omega}\left|\alpha(v)\right|^{p+2}+\left(p+1\right)
\int_{\Omega}\left|\nabla v\right|^{m}\alpha'(v)\left|\alpha(v)\right|^p\\
=\kappa \int_{\Omega}\frac{f(v)}{(\int_{\Omega}
f(v)dx)^{2}}\left|\alpha(v)\right|^p\alpha(v).
\end{multline}
On the other hand, since $\displaystyle{\alpha'(v)\geq\lambda}$, we have
$$
\displaystyle{\int_{\Omega}\left|\nabla v\right|^{m}
\alpha'(v)\left|\alpha(v)\right|^p
\geq\lambda \int_{\Omega}\left|\nabla v\right|^{m}
\left|\alpha(v)\right|^p, \quad \mbox{in }~~ [0,M].}
$$
Now, we discuss two cases.\\

\noindent Case 1. If $\displaystyle{\left|\nabla v \right|
\geq \left|\alpha(v)\right|}$, then
\begin{equation}
\label{w2}
\displaystyle{\int_{\Omega}\left|\nabla v\right|^{m}
\alpha'(v)\left|\alpha(v)\right|^p\geq\lambda 
\int_{\Omega}\left|\alpha(v)\right|^{m+p}.}
\end{equation}

\noindent Case 2. If $\displaystyle{\left|\nabla(v)\right|
\leq \left|\alpha(v)\right|}$, we get
$$
\displaystyle{\int_{\Omega}\left|\nabla v\right|^{m}\alpha'(v)\left|\alpha(v)\right|^p
\geq\lambda \int_{\Omega}\left|\nabla v\right|^{m}\left|\alpha(v)\right|^p
\geq \lambda \int_{\Omega}\left|\nabla v \right|^{m+p}.} 
$$
By using Poincar\'{e}'s inequality, we derive that
$$
\displaystyle{\int_{\Omega}\left|\nabla v\right|^{m}\alpha'(v)
\left|\alpha(v)\right|^p\geq \lambda
\cdot C_{13} \int_{\Omega}\left| v \right|^{m+p}}
\mbox{, for a positive  constant  } C_{13}.	
$$
The smoothness of the function $\alpha$ implies
\begin{equation} 
\label{w1}
\displaystyle{\int_{\Omega}\left|\nabla v\right|^{m}
\alpha'(v)\left|\alpha(v)\right|^p
\geq \dfrac{\lambda\cdot C_{13}}{L_1}
\int_{\Omega} \left|\alpha(v)\right|^{m+p},}
\end{equation}
where $L_1$ is the Lipshitzity constant of function $\alpha$.
Recall from $\displaystyle{{(\ref{L3})}}-\displaystyle{ {(\ref{w1})}}$ that
\begin{multline*}
\dfrac{1}{p+2}\frac{\partial }{\partial s}
\int_{\Omega}\left|\alpha(v)\right|^{p+2}
+\min\left\lbrace \dfrac{\lambda\cdot C_{13}}{L_1},\lambda\right\rbrace
\cdot\int_{\Omega} \left|\alpha(v)\right|^{m+p}\\
\leq\kappa \int_{\Omega}\frac{f(v)}{\left(\int_{\Omega}f(v)dx\right)^{2}}
\left|\alpha(v)\right|^p\alpha(v).
\end{multline*}
It is  easy to check that  
$$
\begin{aligned}
\displaystyle &\dfrac{1}{p+2}\frac{\partial }{\partial s}
\int_{\Omega}\left|\alpha(v)\right|^{p+2}
+\min\left\lbrace \dfrac{\lambda\cdot C_{13}}{L_1},\lambda\right\rbrace
\cdot\int_{\Omega} \left|\alpha(v)\right|^{m+p}
\leq C_{14}\int_{\Omega} \left|\alpha(v)\right|^{p+1},\\ 
&\mbox{for a positive constant } C_{14}.
\end{aligned}
$$
Set $\displaystyle{z_p(s):=\parallel\alpha(v) \parallel_{L^{p+2}(\Omega)}}$
and $C_{15}:=\min\left\lbrace \dfrac{\lambda\cdot C_{13}}{L_1},\lambda\right\rbrace$.
Making use of  H\"{o}lder's inequality and the continuous
embedding of $\displaystyle{L^{m+p}(\Omega)}$ in  $\displaystyle{L^{p+2}(\Omega)}$, 
we obtain that 
$$
\displaystyle{
\frac{\partial z_p(s)}{\partial s}\left(z_p(s)\right)^{p+1}
+C_{15}\left(z_p(s)\right)^{m+p}
\leq C_{14}\left(z_p(s)\right)^{p+1}}.
$$
It follows that
\begin{equation}
\displaystyle{\frac{\partial z_p(s)}{\partial s}
+C_{15}\left(z_p(s)\right)^{m-1}\leq C_{14}.}
\end{equation}
This puts us in a position to employ Ghidaglia's Lemma~\ref{li}, to get
\begin{equation}
\displaystyle{z_p(s)\leq \left( 
\dfrac{ C_{14}}{C_{15}}\right)^{\dfrac{1}{m-1}}
+\dfrac{1}{\left(C_{15}\left(m-2\right)s\right)^{\dfrac{1}{m-2}}}
:=\rho_s.}
\end{equation}
Letting $p$ going to infinity, we obtain that
$$
\parallel\alpha(v) \parallel_{L^{\infty}(\Omega)}\leq C(\eta)
$$
for all $s\geq\eta>0$. This implies 
\begin{equation}
\displaystyle{
\parallel v(s) \parallel_{L^{\infty}(\Omega)}
\leq \max\left(\mid\alpha^{-1}(C(\eta))\mid,
\mid \alpha^{-1}(-C(\eta))\mid\right).}
\end{equation}
Let us consider 
$\displaystyle{\rho:= \max\left(\mid\alpha^{-1}(C(\eta))\mid,
\mid \alpha^{-1}(-C(\eta))\mid\right)}$ as the radius of  
the ball centered at 0. This ball is an absorbing set  
in $\displaystyle{L^{\infty}(\Omega)}$.
\end{proof}

\begin{remark}
Existence of an absorbing set in $\displaystyle{W^{1,m}(\Omega)}$ 
is obtained due to inequality  (\ref{L4}) 
together with the lower semi-continuity
of the norm. It yields that
$$
\displaystyle{ \left|\left|v\left(s\right)\right|\right|_{W^{1,m}(\Omega)}
\leq C(t):=\rho_t,
\quad \mbox{ for all }~~s\geq t.}
$$
Then the ball $\displaystyle{B\left(0,\rho_t\right)}$ 
is an absorbing set in $\displaystyle{W^{1,m}(\Omega)}$.
\end{remark}

Now, in order to prove Lemma~\ref{L6} below, 
we show that the solution of problem (\ref{(P)}) 
is H\"{o}lder continuous. To this end, we set 
$\alpha(v):=w$ and we add the following assumptions:\\ 

\noindent $(H4)$ $\alpha$ is a strict increasing function 
and $\displaystyle{\alpha^{-1}\in \mathcal{C}^1(\mathbb{R})}$;\\

\noindent $(H5)$ $i)$ $\displaystyle{\left(\alpha^{-1}(w)\right)^{\prime}}$ 
is degenerate in the neighborhood of zero and there exists\\ 
$z\in[-\eta_0,\eta_0]$, $\eta_0$ a positive constant, such that
\begin{equation}
\beta_0\left|z\right|^{k_0}\leq\left(\alpha^{-1}(w)\right)^{\prime}
\leq\beta_1\left|z\right|^{k_1}
\end{equation}
for positive constants $\beta_j$ and $k_j$, $j=0,1$;

$ii)$ there exists two positive constants $e_0$ and $e_1$ such that
\begin{equation}
\displaystyle{	e_0
\leq\left(\alpha^{-1}(w)\right)^{\prime}\leq e_1},
\end{equation}
\begin{equation}
\label{L9}
\displaystyle{
\frac{\partial w}{\partial s}-\operatorname{div}
\left(\left|\left(\alpha^{-1}(w)\right)^{\prime}\right|^{m-2}
\cdot\left(\alpha^{-1}(w)\right)^{\prime}|\nabla w|^{m-2} 
\nabla w\right)=\kappa \frac{f(\alpha^{-1}(w))}{
\left(\int_{\Omega}f(\alpha^{-1}(w))dx\right)^{2}},}
\end{equation}
\begin{equation}
\displaystyle{w=0},
\end{equation}
for all $\displaystyle{z\in]-\infty,-\eta_0[~\bigcup~]\eta_0,+\infty[}$.

Identifying (\ref{L9}) with $(1)$ in the paper \cite{vespri1992local}, 
and using hypotheses (H3)--(H5), we can apply the following theorem.

\begin{theorem}[See \cite{vespri1992local}]
\label{L111}
Suppose that  Theorem~\ref{qsdlm}  holds. 
Then, under assumptions (H3)--(H5), the solution 
of problem (\ref{(P)}) is H\"{o}lder continuous.
\end{theorem}

In the following Lemma we prove that the operator
$\displaystyle{\left(S(s)\right)_{s\ge  0}}$  
is uniformly compact for $s$ large enough.

\begin{lemma}
\label{L6}
If $B$ is a bounded set, then 
$$
\displaystyle{
\bigcup_{s \geq s_{0}} S(s) B}
$$ 
is relatively compact
for any $s\geq s_0$.
\end{lemma}

\begin{proof}
We can derive from Lemma~\ref{A3} that the set 
$\bigcup_{s \geq s_{0}} S(s) B$ is bounded in 
$L^\infty(\Omega)$. Furthermore, the approximation 
solution is uniformly bounded. We are in position to invoke 
Theorem~\ref{L111} and, consequently, 
we deduce, by Ascoli--Arzel\`{a} 
theorem, that the set $\displaystyle{
\bigcup_{s \geq s_{0}} S(s) B}
$  is relatively compact.
\end{proof}

\begin{proof}[Proof of Theorem~\ref{L5}]
We have to prove that  $\displaystyle{\left(S(s)\right)_{s\ge 0}}$ 
related to  problem $ {(\ref{(P)})}$ possesses 
an  universal attractor. We consider the following $\omega$-limit: 	
$$
\omega(B_0):= \lbrace v \in L^{\infty}(\Omega): \exists s_n \rightarrow+\infty, 
~\exists v_n \in B_0\mbox{ such that }  
S\left(s_n\right) v_n \rightarrow v
\mbox{ in } L^{\infty}(\Omega)\rbrace,
$$ 
where 
$B_0:=\overline{S(t) B}^{L^{\infty}(\Omega)}$ for some $t>0$.	
We apply Lemma~1.1 in \cite{Temam} to get that $\omega(B_0)$ 
is a nonempty compact invariant set. Then the first condition 
of Definition~\ref{L999} holds. For the second condition 
of Definition~\ref{L999}, we proceed by absurd.  
Assume that $A$ does not attract each bounded set in $L^{\infty}(\Omega)$. 
Then there exists a bounded set $B$, not attracted by $A$, and there exists
$\displaystyle{ s_n\rightarrow\infty}$ and $\epsilon>0$ such that
\begin{equation}
\displaystyle{
dist\left(S(s_n)B,A\right)\geq\dfrac{\epsilon}{2}\mbox{,}}
\end{equation}
from whence follows that,
for every $n$, there exists $ d_n\in B$ such that
\begin{equation}
\label{L8}
\displaystyle{
dist\left(S(s_n)d_n,A\right)\geq\dfrac{\epsilon}{2}.}
\end{equation}
Knowing that $B_0$ is an absorbing set for $B$ (a bounded set), 
there exists $s$ such that $s\geq s_1$, where $s_1$ is a positive constant,  
and we have $S(s)B\subset B_0$. Since  $s_n\rightarrow\infty$, then $s_n\geq s_1$ 
for large enough $n$ and $S(s_n)B\subset B_0$. As a consequence, we have 
\begin{equation}
\displaystyle{
S(s_n)d_n\in  B_0.}
\end{equation}
On the other hand, recall from  Lemma~\ref{L6} that 
$\displaystyle{\bigcup_{s \geq s_{0}} S(s) B_0}$ is relatively compact.
Consequently, the sequence $\displaystyle{\left(S(s_n)d_n\right)_n}$ 
is also relatively compact. So, there exists a subsequence such that
$$
\displaystyle{S(s_n)d_n\longrightarrow 
\ell\in L^{\infty}(\Omega), 
~ \mbox{as}~s_n\longrightarrow \infty.}
$$
With the semi-group propriety, we have
\begin{equation}
\label{qhslds}
\displaystyle{\lim\limits_{n\longrightarrow\infty}S(s_n)d_n
=\lim\limits_{n\longrightarrow\infty} S(s_n-s_1)S(s_1)d_n
=\lim\limits_{n\longrightarrow\infty} S(s'_n)d'_n=\ell,}	
\end{equation}
where $s'_n:=s_n-s_1$ and $d'_n:=S(s_1)d_n$. 
We infer that 
\begin{equation}
\label{RFT1}
\displaystyle{\omega(B_0)
:=\lbrace v:~~\exists s_n,~d_n~\mbox{such that}~~S(s_n)
d_n\longrightarrow v\rbrace.}
\end{equation}
In view of the fact that $\displaystyle{d'_n\in B_0}$, then $\displaystyle{s'_n}$ 
and $\displaystyle{d'_n}$ play the role of $s_n$ and $d_n$, respectively, 
in $\eqref{RFT1}$. Keeping this and \eqref{qhslds} in mind, 
we obtain that $\displaystyle{\ell\in \omega(B_0)=A}$. Then 
$\displaystyle{dist\left(\ell,A\right)=0<\dfrac{\epsilon}{2}}$. 
This is in contradiction with inequality ${(\ref{L8})}$. Hence, 
$A$ is the universal attractor. 
\end{proof}

% -----------------------------------------

\section{Conclusions and perspectives} 
\label{section7}

In this paper, we proved existence and uniqueness of a bounded weak solution 
in Sobolev spaces for a non-local thermistor problem in the presence 
of triply nonlinear terms. We also proved the existence of the global attractor. 
As future work, we  plan to study the regularity of the global attractor, 
the stability of the solution, and the optimal control for the thermistor problem 
\eqref{(P)}.  

% -----------------------------------------

\section*{Acknowledgments}

Torres was supported by FCT through CIDMA and project UIDB/04106/2020.

% -----------------------------------------

% -----------------------------------------

\end{document}